\documentclass[12pt]{amsart}
\usepackage{amsfonts, amssymb}
\usepackage{amsmath}
\usepackage{amsthm}
\usepackage{cite}
\numberwithin{equation}{section}

\newtheorem{theorem}{Theorem}[section]
\newtheorem{lemma}[theorem]{Lemma}

\begin{document}

\title[Hayman's equation]{On transcendental meromorphic solutions of Hayman's equation}

\author{Yueyang Zhang}
\address{School of Mathematics and Physics, University of Science and Technology Beijing, No.~30 Xueyuan Road, Haidian, Beijing, 100083, People's Republic of China}
\email{zhangyueyang@ustb.edu.cn}
\thanks{The author is supported by the Fundamental Research Funds for the Central Universities~{(FRF-TP-19-055A1)}, a Project supported by the National Natural Science Foundation of China~{(12301091)} and the Ministry of Science and Technology of the People's Republic of China~{(G2021105019L)}. The author thanks professor Rod Halburd of the University College London for giving a lot of valuable suggestions during the preparation of this paper. The author would also like to thank professor Wu Chengfa of Shenzhen University for several helpful discussions.}



\date{\today}

\commby{}

\begin{abstract}
We present a complete description of the form of transcendental meromorphic solutions of the second order differential equation
\begin{equation}\tag{\dag}
w''w-w'^2+a w'w+b w^2=\alpha w+\beta w'+\gamma,
\end{equation}
where $a$, $b$, $\alpha$, $\beta$ and $\gamma$ are all rational functions. Together with the Wiman--Valiron theory, we then show that any transcendental meromorphic solution $w$ of equation $(\dag)$ has hyper-order $\varsigma(w)\leq n$ for some integer $n\geq 0$. Moreover, if $w$ has finite order $\sigma(w)$, then $2\sigma(w)$ is a positive integer; if $\beta\equiv\gamma\equiv0$ and $w$ has infinite order or if $\gamma\not\equiv0$ and $w$ has infinite order, then the hyper-order $\varsigma(w)$ is a positive integer.
\end{abstract}

\maketitle


\section{Introduction}\label{Introduction}

In the last several decades, the global properties such as the growth and value distribution of meromorphic solutions of ordinary differential equations (ODEs) have been extensively investigated in the framework of Nevanlinna theory (see~\cite{Laine1993} and references therein). Throughout this paper, a meromorphic function is a complex function meromorphic in the complex plane $\mathbb{C}$. We shall assume that the readers are familiar with the standard notation and basic results of Nevanlinna theory (see also~\cite{Hayman1964}) such as the \emph{characteristic function} $T(r,f)$, the \emph{proximity function} $m(r,f)$, the \emph{counting function} $N(r,f)$ and the \emph{order} $\sigma(f)$ etc. Also recall that the lemma on the logarithmic derivative states that $m(r,f'/f)=O(\log rT(r,f))$, where $r\to\infty$ outside an exceptional set $E$ of finite linear measure (i.e. $\int_{E}dt<\infty$).

An important result due to Gol'dberg~\cite{Gol'dberg1956} states that all meromorphic solutions of the first order ODE: $\Omega(z,f,f')=0$, where $\Omega$ is polynomial in all of its arguments, are of finite order; see also~\cite[Chapter~11]{Laine1993}. A natural question is if there is an upper growth estimate for meromorphic solutions of the second order ODE:
\begin{equation}\label{SDODE}
\Omega(z,f,f',f'')=0,
\end{equation}
where $\Omega$ is polynomial in all of its arguments. In~\cite{Bank1975}, Bank conjectured that the characteristic function $T(r,f)$ for meromorphic solutions of equation \eqref{SDODE} would satisfy $T(r,f)\leq O(\exp(r^c))$ as $r\to \infty$ for some constant $c\geq 0$.
Bank~\cite{Bank1975} himself proved that his conjecture is true with an additional assumption that $N(r,\mu,f)=O(\exp(r^c))$ as $r\to \infty$ for some constant $c\geq 0$ and two distinct values of $\mu\in \mathbb{C}\cup\{\infty\}$. Steinmetz~\cite{Steinmetz1980} proved that Bank's conjecture is true in the special case that \eqref{SDODE} is homogeneous with respect to $f$, $f'$ and $f''$. Steinmetz actually proved that all meromorphic solutions of \eqref{SDODE} in this case take the form $f(z)=[g_1(z)/g_2(z)]e^{g_3(z)}$, where $g_j(z)$, $j=1,2,3$ are entire functions of finite order. In recent years, Bank's conjecture has been proved for some particular second order ODEs \cite{NgWu2019,Conte2015}. However, until now Bank's conjecture still remains open.

Hayman~\cite{Hayman1996} described a generalization of Bank's conjecture for the $n$-th order ODE: \begin{equation}\label{NDODE}
\Omega(z,f,f',\cdots,f^{(n)})=0,
\end{equation}
where $\Omega$ is polynomial in all of its arguments. In the same paper, Hayman used the Wiman--Valiron theory (see \cite{Hayman1974} and also \cite[Chapter~4]{Laine1993}) to study the growth of entire solutions of equation \eqref{NDODE} and provided a condition on the degrees of terms in \eqref{NDODE} under which the entire solutions of \eqref{NDODE} must have finite order. Hayman's theorem \cite[Theorem~C]{Hayman1996} is an extension of the aforementioned Gol'dberg's result. In particular, among the second order algebraic differential equations, Hayman pointed out his theorem does not apply to the equation
\begin{equation}\label{HAYMANEQfu}
a_1(f''f-f'^2)+a_2f'f+a_3f^2+b_1f''+b_2f'+b_3f+b_4=0,
\end{equation}
where $a_i$, $i=1, 2, 3$, $b_j$, $j=1, 2, 3, 4$ are polynomials and $a_1\not\equiv0$. When $a_2\equiv a_3\equiv0$, equation \eqref{HAYMANEQfu} is in some sense \emph{the simplest equation} not covered by Steinmetz's theorem and Hayman's theorem. Hayman conjectured that all meromorphic (entire) solutions of the simplest equation have finite order. This conjecture was confirmed by Chiang and Halburd~\cite{ChiangHalburd2003} in the autonomous case and completely confirmed by Halburd and Wang~\cite{HalburdWang}. In general, if $a_2,a_3$ do not both vanish identically, equation~\eqref{HAYMANEQfu} can have meromorphic solutions of infinite order. For example, the differential equation $f''f-f'^2-f'f=0$ has a solution $f(z)=e^{e^z}$.

One of the main purposes of this paper is to investigate the order of growth of transcendental meromorphic solutions of equation~\eqref{HAYMANEQfu}. Instead of equation \eqref{HAYMANEQfu}, we will actually deal with the second order differential equation
\begin{equation}\label{HAYMANEQ1}
w''w-w'^2+a w'w+b w^2=\alpha w+\beta w'+\gamma,
\end{equation}
where the coefficients $a$, $b$, $\alpha$, $\beta$ and $\gamma$ are all rational functions. Indeed, we may rewrite equation \eqref{HAYMANEQfu} as
\begin{equation}\label{HAYMANEQ2}
ff''-f'^2+\tau_1ff'+\tau_2f^2=\kappa_0+\kappa_1f+\kappa_2f'+\kappa_3f'',
\end{equation}
where $\tau_i$, $i=1,2$, $\kappa_j$, $j=0,1,2,3$ are rational
functions. Then, by letting $w=f-\kappa_3$, equation \eqref{HAYMANEQ2} becomes equation \eqref{HAYMANEQ1} with $a=\tau_1$, $b=\tau_2$, $\alpha=\kappa_1-2\tau_2\kappa_3-\tau_1\kappa_3'-\kappa_3''$,
$\beta=\kappa_2-\tau_1\kappa_3+2\kappa_3'$, $\gamma=\kappa_0+\kappa_1\kappa_3+\kappa_2\kappa_3'+\kappa_3\kappa_3''-\tau_2\kappa_3^2-\tau_1\kappa_3\kappa_3'+\kappa_3'^2-\kappa_3\kappa_3''$.

Halburd and Wang~\cite{HalburdWang} actually found all \emph{admissible} meromorphic solutions of equation \eqref{HAYMANEQ1} with $a\equiv b\equiv 0$ and $\alpha$, $\beta$, $\gamma$ being small functions of $w$. Here and in the following, a small function of $w$, say a meromorphic function $g(z)$, means that $T(r,g)=S(r,w)$, where the notation $S(r,w)$ denotes any quantity satisfying $S(r,w)=o(T(r,w))$, $r\rightarrow{\infty}$, possibly outside an exceptional set of finite linear measure. In particular, all transcendental meromorphic solutions of \eqref{HAYMANEQ1} are admissible when $a$, $b$, $\alpha$, $\beta$ and $\gamma$ are all rational functions. In their proof, Halburd and Wang constructed small functions of $w$ by using the first one or two terms in the local series expansion for $w$ at zeros, which bypasses the issues related to resonance.

In~\cite{zhang2017} the present author considered equation~\eqref{HAYMANEQ1} for the case that $a$ and $b$ are both constants but not necessarily both zero and $\alpha$, $\beta$, $\gamma$ are small functions of $w$. In particular, all nonconstant rational solutions of \eqref{HAYMANEQ1} in the autonomous case were obtained there. When $a,b$ are rational functions, it is in general impossible to list precisely the transcendental meromorphic solutions of equation \eqref{HAYMANEQ1}. In this paper, we shall first give the form of transcendental meromorphic solutions of equation \eqref{HAYMANEQ1} and prove the following

\begin{theorem}\label{maintheorem0}
Suppose that $w$ is a transcendental meromorphic solution of \eqref{HAYMANEQ1}. Then $w$ assumes one of the form described in the following list, where $c_1$, $k_1$ and $k_2$ are constants:
\begin{itemize}
\item[(1)]
If $\beta\equiv\gamma\equiv0$, $\alpha$ is a nonzero constant and $b+2a'+2a^2=0$, then either $w=-\alpha g^2/2$ for a transcendental meromorphic function $g$ satisfying $g'=ag\pm1$ or $e^{\int2adz}$ is a meromorphic function and there is a nonzero constant $k_1$ such that $w=\alpha k_1^{-2}[\cosh(c_1+k_1\int e^{-\int adz}dz)+1]e^{\int2adz}$.

\item[(2)]
If $\gamma\equiv0$, $\beta\not\equiv0$ and $(-\alpha/\beta)'+a(-\alpha/\beta)+b=0$, then $w$ satisfies $w'+(-\alpha/\beta)w=0$.

\item[(3)]
If $\gamma\equiv0$ and $\alpha+\beta'+a\beta\equiv0$, then $w$ satisfies $w'-hw+\beta=0$, where $h$ is a meromorphic function such that $h'+ah+b=0$.

\item[(4)]
If $\gamma\not\equiv0$
and there are two nonzero rational functions $h_1$ and $h_2$ such that $h_1'+a h_1+b=0$, 
$h_2'=(h_1-a)h_2+\alpha+\beta h_1$ and $h_2^2+\beta h_2+\gamma=0$, then $w$ satisfies $w'=h_1w+h_2$.

\item[(5)]
If $\gamma\not\equiv0$ and $A=[\beta(\alpha+\beta')-\gamma'-a(2\gamma-\beta^2)]/\gamma$ satisfies $A'+a A-2b=0$, denoting $B=2\alpha+\beta'+a\beta$, then

\begin{itemize}
\item[(a)]
if $e^{\int 2a dz}$ and $e^{\int \frac{A}{2}dz}$ are meromorphic functions and there are two nonzero constants $k_1,k_2$ such that $B'+2a B+A\alpha+\beta(k_1^2e^{-\int 2a dz}-\frac{A^2}{4}-b)=0$ and
\begin{equation*}
k_2^2=\frac{1}{k_1^2}\left[\frac{1}{4k_1^2}\left(\frac{\beta}{2}A-B\right)^2e^{\int 2a dz}-\left(\frac{\beta^2}{4}-\gamma\right)\right]e^{\int(A+2a)dz},
\end{equation*}
then $w=k_2e^{-\int\frac{A}{2}dz}\cosh (c_1+k_1\int e^{-\int a dz}dz)-\frac{1}{2k_1^2}(\frac{\beta}{2}A-B)e^{\int 2a dz}$; moreover, if $e^{\int 2a dz}$ is rational, then $e^{\int \frac{A}{2}dz}$ is rational and if $e^{\int 2a dz}$ is transcendental, then $\alpha\equiv\beta\equiv0$ and $2(a'+a^2+b)+(\gamma'/\gamma)'+a\gamma'/\gamma=0$;

\item[(b)]
if $e^{\int a dz}$ is a rational function and there is a nonzero constant $k_1$ such that $B'+2a B+A\alpha+\beta(k_1^2e^{-\int 2a dz}-\frac{A^2}{4}-b)=0$ and $k_1^2=\frac{(\frac{\beta}{2}A-B)^2}{\beta^2-4\gamma}e^{\int 2a dz}$, then $w=c_1e^{\int{(-\frac{A}{2}\pm k_1e^{-\int adz})}dz}-\frac{1}{2k_1^2}(\frac{\beta}{2}A-B)e^{\int 2a dz}$;

\item[(c)]
if $e^{\int(\frac{A}{2}+2a)dz}$ is a rational function and there is a nonzero constant $k_1$ such that $k_1=(\frac{\beta}{2}A-B)e^{\int(\frac{A}{2}+2a)dz}$, then $w=(\frac{\beta}{2}A-B)\frac{h^2}{4}-(\frac{\beta^2}{4}-\gamma)\frac{1}{\frac{\beta}{2}A-B}$, where $h$ is a transcendental meromorphic function satisfying $h'=ah+1$;

\item[(d)]
if $e^{\int(\frac{A}{2}+a)dz}$ is a rational function and there is a nonzero constant $k_1$ such that $k_1^2=(\frac{\beta^2}{4}-\gamma)e^{\int(A+2a)dz}$, then $w$ satisfies $w'+\frac{1}{2}(Aw+\beta)=k_1e^{-\int(\frac{A}{2}+a)dz}$;

\item[(e)]
if $\frac{\beta^2}{4}-\gamma\equiv0$, then $w$ satisfies $w'+\frac{1}{2}(Aw+\beta)=0$.

\end{itemize}
\end{itemize}
\end{theorem}

In Theorem~\ref{maintheorem0}, the meromorphic functions $e^{\int 2a dz}$ and $e^{\int \frac{A}{2}dz}$ in part~(1) and part~5(a) and the rational functions $e^{\int a dz}$, $e^{\int(\frac{A}{2}+2a)dz}$ and $e^{\int(\frac{A}{2}+a)dz}$ in parts~(5)(b)-(5)(d) denote fixed meromorphic solutions of certain homogeneous first order differential equations respectively. Note that all poles of $a$ and $A$, if any, are simple. The function $e^{-\int a dz}$ appearing in part~(1) and part~(5)(a) is in general meromorphic on two-sheeted Riemann surface, i.e., an \emph{algebroid} function; see~\cite{Katajamaki1993algebroid} for the theory on algebroid functions. In part~(1), part~(5)(a) and part~(5)(b), $c_1$ denotes the integration constant of the solutions of \eqref{HAYMANEQ1}; in other parts, $k_1$ and $k_2$ denote the integration constants of the solutions of certain differential equations in terms of $a$, $b$, $\alpha$, $\beta$ and $\gamma$ and their derivatives.

When $a$, $b$, $\alpha$, $\beta$ and $\gamma$ are all constants, the conditions in each part of Theorem~\ref{maintheorem0} can be further specified by some simple discussions. For example, in part~(5)(a), we must have $\beta=0$, $A+2a=0$ and $a\alpha=0$. In fact, if $\beta\not=0$, then we easily see that $a=b=A=0$, which then yields a contradiction. Similarly, in part~(5)(b), we have $a=b=0$; in part~(5)(c) and part~(5)(d), we have $A+4a=0$ and $A+2a=0$, respectively. Below we provide some concrete examples for the solutions in part~(4) and part~(5) in the \emph{non-autonomous} case of equation \eqref{HAYMANEQ1}. We shall consider the case $a\not\equiv0$. Note that the assumption $w$ is transcendental would also impose restrictions to the choice of the coefficients $a$, $b$, $\alpha$, $\beta$ and $\gamma$.

For the solutions in part~(4), for example, when $a=z$, $b=-1-z(z+1)$, $\beta=0$, $\alpha$ and $\gamma$ are constants such that $\alpha^2+\gamma=0$, then $h_1=z+1$ and $h_2=-\alpha$. Here we note that $\gamma\not\equiv\beta^2/4$. Otherwise, we have $h_2=-\beta/2$ and deduce from the equation $-h_1h_2+h_2'+ah_2-\alpha-\beta{h_1}\equiv0$ together with the expression of $A$ that $h_1=-A/2$. However, it then follows from $h_1'+ah_1+b=0$ that $A'+aA-2b=0$, a contradiction.

For the solutions in part~(5)(a), when $\beta\not\equiv0$, we suppose that $a=-1/z$, $A=-2sa$, $B=2l+4m/z^4$, $\alpha=l+4m/z^4$, $\beta=m/z^3$, $\gamma=n/z^2$ and $b+s(a'+a^2)=0$, where $l,m,n,s$ are nonzero constants to be determined. In this case, $e^{-\int a dz}=z$ and $e^{-\int\frac{A}{2}dz}=1/z^s$. Then, from the expression of $A$ we have $(A+2a)\gamma=\beta(B-\alpha)-\gamma'$, which gives $2n(s-2)=lm$. We also have the two equations
\begin{equation}\label{coeff 1}
B'+2a B+A\alpha+\beta \left(k_1^2z^2-\frac{s^2}{z^2}+\frac{2s}{z^2}\right)=0
\end{equation}
and
\begin{equation}\label{coeff 2}
k_2^2=\frac{1}{k_1^2}\left[\frac{1}{4k_1^2}\left(sa\beta+B\right)^2\frac{1}{z^2}-\left(\frac{\beta^2}{4}-\gamma\right)\right]z^{2(s-1)}.
\end{equation}
Substitution into the two equations \eqref{coeff 1} and \eqref{coeff 2} shows that $s^2-10s+24=0$, $2sl+mk_1^2=4l$ and, when $s=4$, also that $4l^2+4nk_1^2=0$ and $4k_2^2k_1^2+m^2=0$. Then we can choose $l,m,n,k_1,k_2$ satisfying the equations $4n=lm$, $4l+mk_1^2=0$ and $4k_2^2k_1^2+m^2=0$. Examples of solutions in part~(5)(a) in the case $\beta\equiv0$ will be given in next Section~\ref{order and hyper-order}.

For the solutions in part~(5)(b), when $\beta\not\equiv0$, we suppose that $a=-1/z$, $A=2sa$, $B=2lz^2-2m/z^2$, $\alpha=lz^2$, $\beta=m/z$, $\gamma=n/z^2+tz^2$ and $b=s(a'+a^2)$, where $l,m,n,s,t$ are nonzero constants to be determined. Similarly as before, we may finally find from the expression of $A$ that $(s+2)n=m^2$ and $-2st=m$ and from the other two conditions that $4l=k_1^2m$, $8m-m(s^2+2s)=0$ and, when $s=2$, that $l^2+tk_1^2=0$ and $m^2=4n$. Then we can choose $l,m,n,s,t,k_1$ satisfying the equations $4n=m^2$, $4t+m=0$, $4l=k_1^2m$ and $l^2+tk_1^2=0$. Note that $l=1$. When $\beta\equiv0$, the conditions in Part~(5)(b) read as $A=-2a-\gamma'/\gamma$, $A'+aA-2b=0$, $B=2\alpha$, $2\alpha'+(2a-\gamma'/\gamma)\alpha=0$ and $k_1^2\gamma+\alpha^2e^{\int 2adz}=0$. Then we can choose $a=-1/z$ and suitable $\alpha$ and $\gamma$.

For the solutions in parts~(5)(c)-(5)(e), it is easy to choose suitable $a$, $b$, $\alpha$, $\beta$ and $\gamma$ satisfying the conditions there. Here we point out that in part~(5)(c) we must have $\beta\not\equiv0$. Otherwise, we have $A=-2a-\gamma'/\gamma$, $A'+aA-2b=0$, $B=2\alpha$, $2\alpha'+(2a-\gamma'/\gamma)\alpha=0$ and $k_1^2+2\alpha e^{\int (\frac{A}{2}+2a)dz}=0$. Since $k_1\not=0$, we see that $\alpha\not\equiv0$ and thus $a\to0$ as $z\to\infty$. But then the equation $H'=aH+1$ cannot have transcendental meromorphic solutions.

It is difficult, if not impossible, to further determine the form of the integral $\int e^{-\int adz}dz$ in part~(1) and part~(5)(a), but we note that solution $w$ with an algebroid function $e^{-\int a dz}$ can indeed occur. Consider the solutions in part~(5)(a). When $a=-1/(2z)-1$, $\alpha=\beta=B=0$ and $A=-\gamma'/\gamma-2a$, if we choose $\gamma$ so that $e^{\int\frac{A}{2}dz}$ is a meromorphic function and $g=\sqrt{z}\varphi$ with an entire function $\varphi$ satisfying $2z\varphi'+\varphi=2ze^z$, then $g'=\sqrt{z}e^z=e^{-\int a dz}$ and, when $c_1=k_1=1$, $\cosh g$ is an entire function. Solution $w$ with  an algebraic function $e^{\int a dz}$ can also occur, as will be seen in Section~\ref{order and hyper-order} below.

The rest of this paper is structured as follows. In Section~\ref{order and hyper-order}, based on the results in Theorem~\ref{maintheorem0}, we use the Wiman--Valiron theory to investigate the order of growth of transcendental meromorphic solutions of \eqref{HAYMANEQ1}. In Section~\ref{Transcendental meromorphic solutions}, we present a proof for Theorem~\ref{maintheorem0}. Since now $a$ and $b$ are assumed to be rational functions, to derive the form of $w$ in part~(1) and part~(5) we need to first introduce some auxiliary functions which are defined and meromorphic in certain part of the plane. The meromorphicity of $w$ and $a$, $b$, $\alpha$, $\beta$ and $\gamma$ together with certain identities finally imply that the functions $e^{\int 2a dz}$ and $e^{\int \frac{A}{2}dz}$ in part~(1) and part~5(a), as well as the functions $e^{\int a dz}$, $e^{\int(\frac{A}{2}+2a)dz}$ and $e^{\int(\frac{A}{2}+a)dz}$ in parts~(5)(b)-(5)(d), are meromorphic functions in the plane.

\section{Growth of meromorphic solutions of equation \eqref{HAYMANEQ1}}\label{order and hyper-order}

In this section, we investigate the order of growth of transcendental meromorphic solutions of equation \eqref{HAYMANEQ1}. When $a\equiv b\equiv0$, all transcendental meromorphic solutions of \eqref{HAYMANEQ1} are of exponential type and of order one~\cite[Corollary~1.2]{HalburdWang}. When $a,b$ do not both vanish identically, the orders of transcendental meromorphic solutions of equation \eqref{HAYMANEQ1} have more possibilities. In particular, the solutions may have infinite order. For a meromorphic function $g(z)$ of infinite order, define the \emph{hyper-order} $\varsigma(g)$ of $g(z)$ as
\begin{equation*}
\varsigma(g)=\limsup_{r\to\infty}\frac{\log\log T(r,g)}{\log r}.
\end{equation*}
Based on the results in Theorem~\ref{maintheorem0}, we prove the following

\begin{theorem}\label{maintheorem1}
Suppose that $w$ is a transcendental meromorphic solution of equation \eqref{HAYMANEQ1}. Then $\varsigma(w)\leq n$ for some integer $n\geq 0$. Moreover,
\begin{itemize}
\item [(1)] If $w$ has finite order $\sigma(w)$, then $2\sigma(w)$ is a positive integer;
\item [(2)] If $\beta\equiv\gamma\equiv0$ and $w$ has infinite order or if
            $\gamma\not\equiv0$ and $w$ has infinite order, then $\varsigma(w)$ is a positive integer.
\end{itemize}

\end{theorem}

With certain choice of the coefficients $a$, $b$, $\alpha$, $\beta$ and $\gamma$, equation \eqref{HAYMANEQ1} can have meromorphic solutions with order $\sigma(w)=n/2$ or hyper-order $\varsigma(w)=n$ for any integer $n\geq 1$. Look at the solution $w$ in Theorem~\ref{maintheorem0}~(5)(a). Let $N\leq 1$ be an integer. If we choose $\beta=0$, $\gamma=z^{N}$, $a=Nz^{-1}/2$, $A+2a=-\gamma'/\gamma=-2a$, $B=2\alpha$ is a constant and $k_1^4k_2^2\gamma=\alpha^2z^N+k_1^2\gamma$, then $e^{\int 2a dz}$ is rational and $w$ has order $\sigma(w)=n/2$, $n=2-N$, with suitable $\alpha,k_1,k_2$. On the other hand, if we choose $a$ to be a polynomial of degree $n-1$ and thus $e^{\int 2a dz}$ is transcendental, then by choosing $\gamma=P^2$ for some polynomial $P=P(z)$, the equations $\alpha\equiv\beta\equiv0$, $A+2a=-\gamma'/\gamma$, $k_1^2k_2^2=1$ and $2(a'+a^2+b)+(\gamma'/\gamma)'+a\gamma'/\gamma=0$ can be satisfied and $w$ has hyper-order $\varsigma(w)=n$.

By Theorem~\ref{maintheorem1}, Bank's conjecture holds for equation \eqref{HAYMANEQ1}. By the proof of Theorem~\ref{maintheorem1}, solutions $w$ with infinite order occur in part~(1), part~(3) and part~(5)(a) of Theorem~\ref{maintheorem0} and, moreover, $a(z)\not\to 0$ as $z\to\infty$ in each of three cases. However, for the solutions in Theorem~\ref{maintheorem0}~(3), when $\beta\not\equiv0$, the exact order or hyper-order of $w$ satisfying the equation $w'-hw+\beta=0$, where $h$ is a transcendental meromorphic function such that $h'+ah+b=0$, is not determined. This problem can be viewed as a generalization of Br\"{u}ck's conjecture in the uniqueness theory of meromorphic functions~\cite{Bruck1996}. When $a$, $b$ and $\beta$ are constants, this problem has been completely resolved in \cite{zhang2024}.

For a transcendental entire function $g(z)$, we have the Taylor series $g(z)=\sum_{n=0}^{\infty}a_nz^n$. Denote by $M(r,g)=\max_{|z|=r}|g(z)|$ the \emph{maximum modulus} of $g(z)$ on the circle $|z|=r>0$ and by $\nu(r,g)$ the \emph{central index} of $g(z)$, which is defined as the greatest exponent of the maximal term in the Taylor series of $g(z)$. Basically, $\nu(r,g)$ is increasing, piecewise constant, right-continuous and $\nu(r,g)\to\infty$ as $r\to\infty$. Moreover,
\begin{equation*}
\sigma(g)=\limsup_{r\to\infty}\frac{\log \nu(r,g)}{\log r}.
\end{equation*}
See, e.g., \cite[Theorem~3.1]{Laine1993}. We also have
\begin{equation*}
\varsigma(g)=\limsup_{r\to\infty}\frac{\log\log \nu(r,g)}{\log r}.
\end{equation*}
The following is the Wiman--Valiron theorem (see e.g. \cite[Theorem~3.2]{Laine1993}).

\begin{lemma}[see \cite{Laine1993}]\label{le wm1}
Let $g$ be a transcendental entire function, let $0<\delta<1/4$ and $z$ be such that $|z|=r$ and
\begin{equation*}
|g(z)|>M(r,g)\nu(r,g)^{-1/4+\delta}
\end{equation*}
holds. Then there exists a set $F\subset \mathbb{R}_{+}$ of finite logarithmic measure, i.e., $\int_{F}dt/t<\infty$, such that
\begin{equation*}
g^{(m)}(z)=\left(\frac{\nu(r,g)}{z}\right)^m(1+o(1))g(z),
\end{equation*}
for all $m\geq 0$ and all $r\not\in F$.
\end{lemma}

\begin{proof}[Proof of Theorem~\ref{maintheorem1}]

The solutions $w$ in Theorem~\ref{maintheorem0}~(1) and Theorem~\ref{maintheorem0}~(5)(a) have similar orders or hyper-orders. They will be considered in the end of the proof.

For the solution $w$ in Theorem~\ref{maintheorem0}~(2), we may integrate the differential equation $w'+(-\alpha/\beta)w=0$ and obtain that $w$ is of the form $w=u(z)e^{v(z)}$, where $u(z)$ is a rational function and $v(z)$ is a polynomial. Thus $w$ has positive integer order.

The solution $w$ in Theorem~\ref{maintheorem0}~(3) may have infinite order. Recall that $w$ and $h$ satisfy the two differential equations
\begin{equation}\label{order1}
\begin{split}
w'=hw-\beta, \qquad h'=-ah-b.
\end{split}
\end{equation}
Since $a$, $b$ and $\beta$ are rational functions, $w$ and $h$ both have at most finitely many poles. Then there are two nonzero polynomials $\omega_1$ and $\omega_2$ such that $W=w\omega_1$ and $H=h\omega_2$ are both entire functions. By substituting $w=W/\omega_1$ and $h=H/\omega_2$ into the equations in \eqref{order1} respectively, we obtain
\begin{equation}\label{order2}
\begin{split}
W'=\left(\frac{H}{\omega_2}+\frac{\omega_1'}{\omega_1}\right)W-\beta\omega_1, \qquad
H'=-\left(a-\frac{\omega_2'}{\omega_2}\right)H-b\omega_2.
\end{split}
\end{equation}
Write
\begin{equation}\label{order3}
a(z)=\sum_{i=0}^{m}\frac{\mu_i}{(z-\nu_i)^{n_i}}+Q(z),
\end{equation}
where $m,n_i$ are nonnegative integers, $\mu_i,\nu_j$ are constants and $Q(z)$ is a polynomial. We distinguish the two cases whether or not $Q(z)$ in \eqref{order3} vanishes identically.

If $Q(z)$ in \eqref{order3} vanishes identically, then by integrating the second differential equation in \eqref{order2} we easily deduce that $H$ is a polynomial. Similarly, since $W$ is transcendental, we must have $H(z)/\omega_2(z)=\eta_1 z^{m_1}(1+o(1))$ as $z\to\infty$, where $m_1\geq 0$ is an integer and $\eta_1$ is a nonzero constant. Now we choose $z=re^{i\theta}$, $\theta\in[0,2\pi)$ and $r>0$ such that $|W(z)|=M(r,W)$. It follows that $\beta(z)\omega_1(z)/W(z)=o(1)\cdot z^{-k}$ as $z\to\infty$ for a large positive integer $k$. We apply Lemma~\ref{le wm1} to $W$ and divide both sides of the first equation in \eqref{order2} by $W$ to deduce from the resulting equation that
\begin{equation*}
\frac{\nu(r,W)}{z}\left(1+o(1)\right)=\eta_1 z^{m_1}\left(1+o(1)\right)+\frac{\omega_1'(z)}{\omega_1(z)}+o(1)\cdot z^{-k},
\end{equation*}
where $r\to\infty$ outside an exceptional set of finite logarithmic measure. This yields $\nu(r,W)=\eta_1r^{m_1+1}(1+o(1))$. It is standard to obtain from this estimate that $\sigma(W)=m_1+1$; see e.g. \cite[pp.~74-75]{Laine1993}. Thus $\sigma(w)=m_1+1$ is a positive integer. The above estimates actually show that $\sigma(w)$ is a positive integer when $h(z)$ is a rational function, regardless whether $Q(z)$ in \eqref{order3} vanishes identically or not.

If $Q(z)$ in \eqref{order3} does not vanish identically and $h$ is transcendental, denoting the degree of $Q(z)$ by $m_2$, then by the same arguments as before we may obtain $\sigma(h)=\sigma(H)=m_2+1$. Moreover, by integrating the second differential equation in \eqref{order2}, we easily deduce that $|H(z)|\leq \exp(\eta_2 r^{\sigma(h)})$ for all $r>0$ and some positive constant $\eta_2$. Again, by applying Lemma~\ref{le wm1} to $W$, we deduce from the first differential equation in \eqref{order2} that
\begin{equation*}
\frac{\nu(r,W)}{z}\left(1+o(1)\right)=\left(\frac{H(z)}{\omega_2(z)}+\frac{\omega_1'(z)}{\omega_1(z)}\right)+o(1)\cdot z^{-k},
\end{equation*}
where $r\to\infty$ outside an exceptional set of finite logarithmic measure. This yields $\nu(r,W)\leq \exp(2\eta_2 r^{\sigma(h)})$ for all $r$ outside an exceptional set of finite logarithmic measure. By using Borel's lemma (see \cite[Lemma~1.1.2]{Laine1993}) to remove the exceptional set, we get $\varsigma(w)\leq \sigma(h)=m_2+1$. In particular, if $\beta\equiv0$ and thus $\alpha\equiv0$, by the lemma on the logarithmic derivative we have
\begin{equation*}
T(r,h)=m(r,h)+O(\log r)=m\left(r,\frac{w'}{w}\right)+O(\log r)=O(\log rT(r,w)),
\end{equation*}
where $r\to\infty$ outside an exceptional set of finite linear measure. It follows that $T(r,h)\leq \eta_3\log(rT(r,w))$ for some positive constant $\eta_3$ and all $r$ outside an exceptional set of finite linear measure. By using Borel's lemma (see \cite[Lemma~1.1.1]{Laine1993}) to remove the exceptional set, we also have $m_2+1\leq \varsigma(w)$. Thus $\varsigma(w)=m_2+1$.

For the solution $w$ in Theorem~\ref{maintheorem0}~(4), since there are two nonzero rational functions $h_1$ and $h_2$ such that $w'=h_1w+h_2$, $w$ has positive integer order.

Now look at the solution $w$ in Theorem~\ref{maintheorem0}~(5). From the proof of Theorem~\ref{maintheorem0} in Section~\ref{Transcendental meromorphic solutions} we know that $w$ satisfies the differential equation in \eqref{EQ10}, i.e.,
\begin{equation}\label{order gro1}
\left(w'+\frac{1}{2}[Aw+\beta]\right)^2=\left(g+\frac{A^2}{4}\right)w^2+\left(\frac{\beta}{2}A-B\right)w+\left(\frac{\beta^2}{4}-\gamma\right),
\end{equation}
where $g$ is a meromorphic function with at most finitely many poles and $A$ and $B$ are both rational functions. Then there is a nonzero polynomial $\omega_3$ such that $W=w\omega_3$ is an entire function. By substituting $w=W/\omega_3$ into \eqref{order gro1} and then dividing both sides of the resulting equation by $w^2$ and then applying Lemma~\ref{le wm1} to $W$, we finally obtain
\begin{equation}\label{order gro3}
\left[\frac{\nu(r,W)}{z}\left(1+o(1)\right)+\frac{1}{2}A(z)-\frac{\omega_3'(z)}{\omega_3(z)}+o(1)\cdot z^{-k}\right]^2=g(z)+\frac{A(z)^2}{4}+o(1)\cdot z^{-k},
\end{equation}
where $r\to\infty$ outside an exceptional set of finite logarithmic measure. Thus, for the solution $w$ in Theorem~\ref{maintheorem0}~(5)(b)-(5)(e), since $g+A^2/4\equiv0$, by the properties of $\nu(r,W)$ we see that $A(z)/2-\omega_3'(z)/\omega_3(z)\not\to 0$ as $z\to\infty$. Then, by the same arguments as before, we easily obtain from \eqref{order gro3} that $W$, and hence $w$, has positive integer order.

Consider the solution $w$ in Theorem~\ref{maintheorem0}~(5)(a). From the proof of Theorem~\ref{maintheorem0} we know that $g+A^2/4=k_1^2e^{-\int 2adz }$ is a meromorphic function and $k_1$ is a nonzero constant. Then $g+A^2/4$ must be of the form $u_1(z)e^{v_1(z)}$, where $u_1(z)$ is a rational function and $v_1(z)$ is a polynomial. If $e^{\int2adz }$ is a rational function, then $v_1(z)$ is a constant and $a(z)\to 0$ as $z\to\infty$. Also, since $e^{\int\frac{A}{2}dz}$ is a rational function, $A(z)\to 0$ as $z\to\infty$. By the properties of $\nu(r,W)$, we see from \eqref{order gro3} that $zu_1(z)\not\to 0$ as $z\to\infty$ and thus
\begin{equation}\label{order gro4}
\frac{\nu(r,W)}{z}=\eta_4 z^{\frac{m_3}{2}}\left(1+o(1)\right),
\end{equation}
where $m_3\geq -1$ is an integer and $\eta_4$ is nonzero constant. This implies that $W$, and hence $w$, has finite order $(m_3+2)/2$ and thus $2\sigma(w)=m_3+2$ is an integer.

If $e^{\int2adz }$ is transcendental, then $v_1(z)$ is a nonconstant polynomial of degree $m_4\geq 1$. Note that $T(r,g+A^2/4)=\eta_5r^{m_4}(1+o(1))$ as $r\to\infty$ for some positive constant $\eta_5$. Then we have from \eqref{order gro3} that $\nu(r,W)\leq \exp(\eta_6 r^{m_4})$, where $\eta_6$ is a positive constant, for all $r$ outside an exceptional set of finite logarithmic measure. This yields $\varsigma(w)=\varsigma(W)\leq m_4$. On the other hand, since $\gamma\not\equiv0$, from the proof of Theorem~\ref{maintheorem0} in Section~\ref{Transcendental meromorphic solutions} below we know that $m(r,1/w)=O(\log rT(r,w))$. Together with the lemma on the logarithmic derivative, we divide both sides of the equation in \eqref{order gro1} by $w^2$ and then deduce from the resulting equation that
\begin{equation*}
\begin{split}
&\ T\left(r,g+\frac{A^2}{4}\right)=m\left(r,g+\frac{A^2}{4}\right)+O(\log r)\\
=&\ m\left(r,\left(\frac{w'}{w}+\frac{1}{2}A+\frac{1}{2}\frac{\beta}{w}\right)^2-\left(\frac{\beta}{2}A-B\right)\frac{1}{w}-\left(\frac{\beta^2}{4}-\gamma\right)\frac{1}{w^2}\right)+O(\log r)\\
\leq&\
2m\left(r,\frac{w'}{w}\right)+5m\left(r,\frac{1}{w}\right)+O(\log r)=O(\log rT(r,w)),
\end{split}
\end{equation*}
where $r\to\infty$ outside an exceptional set of finite linear measure. Then, by similar arguments as before, we obtain from the above estimate that $m_4\leq \varsigma(w)$. Thus $\varsigma(w)=m_4$.

Finally, look at the solutions in Theorem~\ref{maintheorem0}~(1). If $w=-\alpha g^2/2$ for a transcendental meromorphic function $g$ satisfying $g'=ag\pm1$, then $g$, and hence $w$, has positive integer order. Otherwise, from the proof of Theorem~\ref{maintheorem0} in Section~\ref{Transcendental meromorphic solutions} we know that the solution $w$ in Theorem~\ref{maintheorem0}~(1) satisfies the differential equation in \eqref{EQ1}, which can be rewritten as
\begin{equation}\label{order gro5}
\left(\frac{w'}{w}-2a\right)^2+2\frac{\alpha}{w}=f(z)+3a^2=k_1^2e^{-\int2adz},
\end{equation}
where $k_1$ is a nonzero constant and $e^{-\int2adz}$ is a meromorphic function. Also, we have $m(r,1/w)=O(\log rT(r,w))$. Then by exactly the same arguments as for the solutions in Theorem~\ref{maintheorem0}~(5)(a), we can show that $2\sigma(w)$ is a positive integer when $e^{-\int2adz}$ is rational or $\varsigma(w)$ is a positive integer when $e^{-\int2adz}$ is transcendental.

By summing the above estimates together with the results in Theorem~\ref{maintheorem0}, we obtain the desired two assertions in Theorem~\ref{maintheorem1}.

\end{proof}

\section{Transcendental meromorphic solutions of equation \eqref{HAYMANEQ1}}\label{Transcendental meromorphic solutions}

\begin{proof}[Proof of Theorem~\ref{maintheorem0}]

Let $w$ be a transcendental meromorphic solution of equation \eqref{HAYMANEQ1}. If $\alpha\equiv\beta\equiv\gamma\equiv0$, then equation \eqref{HAYMANEQ1} becomes $(w'/w)'+a(w'/w)+b=0$. This is the special case of part (3) of Theorem~\ref{maintheorem0}. From now on, we suppose that at least one of $\alpha$, $\beta$ and $\gamma$ is nonzero.

Define the set $\Phi_f$ for any meromorphic function $f$ as follows: if $f\equiv0$, then $\Phi_f=\emptyset$; if $f\not\equiv0$, then $\Phi_f$ denotes the set of all zeros and poles of $f$. Let $\Phi=\Phi_{a}\cup\Phi_{b}\cup\Phi_{\alpha}\cup\Phi_{\beta}\cup\Phi_{\gamma}$. Then $\Phi$ contains at most finitely many points. Let $z_0\in\Psi:=\mathbb{C}\backslash\Phi$ be either a zero or a pole of $w$. Then, in a neighborhood of $z_0$, $w$ has a Laurent series expansion of the form $w(z)=a_0\xi^p+a_1\xi^{p+1}+O(\xi^{p+2})$, where $\xi=z-z_0$, $a_0\not=0$, $a_1$ are constants and $p$ is a nonzero integer. Substitution into \eqref{HAYMANEQ1} gives
\begin{equation*}
\begin{split}
&\left(-pa_0^2\xi^{2p-2}+\cdots\right)+\left(a(z_0)+\cdots\right)\left(pa_0^2\xi^{2p-1}+\cdots\right)+\left(b(z_0)+\cdots\right)\left(a_0^2\xi^{2p}+\cdots\right)\\
=&\,\left(\alpha(z_0)+\cdots\right)\left(a_0\xi^p+\cdots\right)+\left(\beta(z_0)+\cdots\right)\left(a_0p\xi^{p-1}+\cdots\right)+\left(\gamma(z_0)+\cdots\right).
\end{split}
\end{equation*}
It follows that $p=2$ if $\beta\equiv\gamma\equiv0$ and $p=1$ in other cases. In particular, we see that $w$ has at most finitely many poles and is analytic on $\Psi$. As in \cite{HalburdWang,zhang2017}, we distinguish three cases: (1) $\alpha\not\equiv0$, $\beta\equiv\gamma\equiv0$; (2) $\beta\not\equiv0$, $\gamma\equiv0$; (3) $\gamma\not\equiv0$.

\vskip 3pt

\noindent\textbf{Case 1:} $\alpha\not\equiv0$, $\beta\equiv\gamma\equiv0$.

\vskip 3pt

By substituting $w(z)=a_0\xi^2+a_1\xi^3+O(\xi^4)$ into \eqref{HAYMANEQ1}, about any $z_0\in\Psi$ such that $w(z_0)=0$, we have
\begin{equation*}
w(z)=-\frac{\alpha(z_0)}{2}\xi^2-\left(\frac{\alpha'(z_0)+a(z_0)\alpha(z_0)}{2}\right)\xi^3+O(\xi^4).
\end{equation*}
Since $w$ is analytic on $\Psi$, it follows that
\begin{equation}\label{EQ1}
f(z):=\left(\frac{w'}{w}-\frac{\alpha'+a\alpha}{\alpha}\right)^2+2\frac{\alpha}{w}-2\frac{\alpha'+a\alpha}{\alpha}\left(\frac{w'}{w}\right)
\end{equation}
is also analytic on $\Psi$. (The definition of $f(z)$ in \cite{zhang2017} in this case is incorrect). By the definition of $\Psi$, we have $N(r,f)=O(\log r)$. We write equation \eqref{HAYMANEQ1} with $\beta\equiv\gamma\equiv0$ as
\begin{equation}\label{EQ2}
\frac{1}{w}=\frac{1}{\alpha}\left[\left(\frac{w'}{w}\right)'+a\frac{w'}{w}+b\right].
\end{equation}
By taking the proximity functions on both sides of equation \eqref{EQ2} together with the lemma on the logarithmic derivative, we may deduce that $m(r,1/f)=O(\log rT(r,w))$. Then, together with the lemma on the logarithmic derivative, we may obtain from \eqref{EQ1} that $m(r,f)=O(\log rT(r,w))$. Hence $T(r,f)=S(r,w)$, i.e., $f$ is a small function of $w$.

Differentiating \eqref{EQ1} and using \eqref{HAYMANEQ1} to eliminate $(w'/w)'$ from the resulting equation and then using \eqref{EQ1} to eliminate $(w'/w)^2$ from the resulting equation gives
\begin{equation}\label{EQ3}
\phi w'=\varphi w+2\alpha',
\end{equation}
where
\begin{equation*}
\begin{split}
\phi&=2\left[b+2a\frac{\alpha'+a\alpha}{\alpha}+2\left(\frac{\alpha'+a\alpha}{\alpha}\right)'\right],\\
\varphi&=\left[4b\frac{\alpha'+a\alpha}{\alpha}+2\left(\frac{\alpha'+a\alpha}{\alpha}\right)'\frac{\alpha'+a\alpha}{\alpha}+2a\left(\frac{\alpha'+a\alpha}{\alpha}\right)^2\right]-f'-2af.
\end{split}
\end{equation*}

\vskip 3pt

\noindent\textbf{Case 1a:} $\phi\equiv0$.

\vskip 3pt

When $\phi\equiv0$, since $a$, $b$ and $\alpha$ are all rational functions and $f$ is a small function of $w$, by \eqref{EQ3} we must have $\alpha'\equiv 0$ and $\varphi\equiv0$. Thus $\alpha$ is a constant, $b+2a'+2a^2\equiv0$ and also
\begin{equation*}
\begin{split}
f'+2af=4ba+2aa'+2a^3=-6a\left(a'+a^2\right).
\end{split}
\end{equation*}
Denote $F=f+3a^2$. Then $F$ satisfies the differential equation $F'+2aF=0$. Obviously, $F=0$ is a solution of this equation. If $F\equiv0$, then equation \eqref{EQ1} becomes
\begin{equation}\label{EQ4}
\begin{split}
\left(w'-2aw\right)^2+2\alpha w=0.
\end{split}
\end{equation}
By writing $w=-\alpha g^2/2$, where $g$ is a meromorphic function, it follows from \eqref{EQ4} that $g$ satisfies the equation $g'=ag\pm1$. Otherwise, we may write $F=f+3a^2=k_1^2e^{-\int 2adz}$, where $k_1$ is a nonzero constant and $e^{-\int2adz}$ is a meromorphic function with at most finitely many zeros and poles. In terms of $u=k_1^{2}we^{-\int2adz}$, equation \eqref{EQ1} becomes
\begin{equation*}
\begin{split}
u'^2-k_1^2e^{-\int 2adz}u^2+2k_1^2\alpha ue^{-\int2adz}=0,
\end{split}
\end{equation*}
which can be rewritten as
\begin{equation*}
\begin{split}
\left(\frac{u'}{k_1e^{-\int adz}}+(u-\alpha)\right)\left(\frac{u'}{k_1e^{-\int adz}}-(u-\alpha)\right)=-\alpha^2.
\end{split}
\end{equation*}
Here $e^{\int a dz}$ is in general an algebroid function and has at most finitely many algebraically branched points in the plane, as mentioned in the introduction. We may fix one branch of $e^{\int a dz}$. Denote $\kappa:=u'/(k_1e^{-\int adz})+(u-\alpha)$. From the above equation we have $u'/(k_1e^{-\int adz})-(u-\alpha)=-\alpha^2\kappa^{-1}$ and further that
\begin{equation}\label{EQ4 e}
\begin{split}
u=\frac{1}{2}\left(\kappa+\alpha^2\kappa^{-1}\right)+\alpha,  \qquad
u'=\frac{1}{2}\left(\kappa-\alpha^2\kappa^{-1}\right)k_1e^{-\int adz}.
\end{split}
\end{equation}
By taking the derivatives on both sides of the first equation in \eqref{EQ4 e} and then comparing the resulting equation with the second equation in \eqref{EQ4 e}, we find $\kappa'=k_1e^{-\int adz}\kappa$ and thus, by integration, $\kappa=\alpha e^{c_1+k_1\int e^{-\int adz}dz}$, where $c_1$ is a constant. By combining the above results together we get $w=\alpha k_1^{-2}[\cosh(c_1+k_1\int e^{-\int adz}dz)+1]e^{\int2adz}$.

\vskip 3pt

\noindent\textbf{Case 1b:} $\phi\not\equiv0$.

\vskip 3pt

We get from \eqref{EQ3} that $w'=p_1w+p_2$, where $p_1=\varphi/\phi$, $p_2=2\alpha'/\phi$, and it follows that $w''=(p_1'+p_1^2)w+p_1p_2+p_2'$. Substituting these results into
\eqref{HAYMANEQ1} with $\beta\equiv\gamma\equiv0$, we get
\begin{equation}\label{EQ5-fu}
\left(p_1'+ap_1+b\right)w^2+\left(-p_1p_2+p_2'+ap_2-\alpha\right)w+p_2^2=0
\end{equation}
with coefficients that are small functions of $w$. The Valiron--Mohon'ko theorem (see, e.g. \cite[Theorem~2.2.5]{Laine1993}) implies that all the coefficients in
\eqref{EQ5-fu} must vanish identically. Thus $-p_1p_2+p_2'+ap_2-\alpha=0$ and $p_2=0$. But this yields $\alpha\equiv0$, a contradiction.

\vskip 3pt

\noindent\textbf{Case 2:} $\beta\not\equiv0$, $\gamma\equiv0$.

\vskip 3pt

Recall that in this case $w$ has only simple zeros in $\Psi$. Substituting $w(z)=a_0\xi+a_1\xi^2+O(\xi^3)$ into \eqref{HAYMANEQ1}, we find at the leading order $a_0=-\beta(z_0)$ and at the next-to-leading order that $\alpha(z_0)+\beta'(z_0)+a(z_0)\beta(z_0)=0$.

\vskip 3pt

\noindent\textbf{Case 2a:} $\alpha+\beta'+a\beta\not\equiv{0}$, $\gamma\equiv0$.

\vskip 3pt

Let $f=w'/w$. If $z_0$ is a pole of $w$, then $z_0\in \Phi$; if $z_0$ is a zero of $w$, then either $z_0\in\Phi$ or $\alpha(z_0)+\beta'(z_0)+a(z_0)\beta(z_0)=0$. Since $a$, $b$, $\alpha$ and $\beta$ are all rational functions, by the definition of $\Phi$ we have $N(r,f)=O(\log r)$. Moreover, an application of the lemma on the logarithmic derivative shows that $m(r,f)=O(\log rT(r,w))$. Hence $T(r,f)=S(r,w)$. Substituting $w'=fw$ and $w''=(f'+f^2)w$ into \eqref{HAYMANEQ1} with $\gamma\equiv0$, we get
\begin{equation*}
\left(f'+a f+b\right)w=\alpha+\beta f.
\end{equation*}
Since $f'+a f+b$ and $\alpha+f\beta$ are both small functions of $w$, we must have
$f'+a f+b\equiv 0$ and $\alpha+\beta f\equiv0$. Thus $f=-\alpha/\beta$ and
$(-\alpha/\beta)'+a(-\alpha/\beta)+b=0$. These results give part~(2) of the theorem.

\vskip 3pt

\noindent\textbf{Case 2b:} $\alpha+\beta'+a\beta\equiv{0}$, $\gamma\equiv0$.

\vskip 3pt

Equation \eqref{HAYMANEQ1} takes the form $((w'+\beta)/w)'+a((w'+\beta)/w)+b=0$. By denoting $h=(w'+\beta)/w$, we have part~(3) of the theorem.

\vskip 3pt

\noindent\textbf{Case 3:} $\gamma\not\equiv0$.

\vskip 3pt

Recall that in this case $w$ is analytic in $\Psi$ and any zero $z_0$ of $w$ in $\Psi$ is
simple. On substituting $w(z)=a_0\xi+a_1\xi^2+O(\xi^3)$ into \eqref{HAYMANEQ1}, we find that $a_0^2+\beta(z_0)a_0+\gamma(z_0)=0$ and $a_1=\delta_1(z_0)a_0-\delta_2(z_0)$, where $\delta_1=[\gamma'+a(\gamma-\beta^2)-\beta(\alpha+\beta')]/(2\gamma)$ and $\delta_2=(\alpha+\beta'+a\beta)/2$ are two rational functions. Denote
\begin{equation}\label{EQ5}
\begin{split}
A=\frac{\beta(\alpha+\beta')-\gamma'-a(2\gamma-\beta^2)}{\gamma},\qquad
B=2\alpha+\beta'+a\beta
\end{split}
\end{equation}
and let
\begin{equation}\label{EQ6}
g(z)=\frac{w'^2+\beta{w'}+\gamma}{w^2}+A\frac{w'}{w}+B\frac{1}{w}.
\end{equation}
Then $g(z)$ is analytic in $\Psi$. Since $a$, $b$, $\alpha$, $\beta$ and $\gamma$ are all rational functions, by the definition of $\Psi$ we have $N(r,g)=O(\log r)$. Moreover, rewriting \eqref{HAYMANEQ1} as
\begin{equation}\label{EQ7}
\frac{1}{w^2}=\frac{1}{\gamma}\left[\left(\frac{w'}{w}\right)'+a\frac{w'}{w}+b-\frac{1}{w}\left(\alpha+\beta\frac{w'}{w}\right)\right]
\end{equation}
and then taking the proximity functions on both sides of \eqref{EQ7} together with the lemma on the logarithmic derivative gives $2m(r,1/w)\leq m(r,1/w)+O(\log rT(r,w))$ and so $m(r,1/w)=O(\log rT(r,w))$. Then, by the lemma on the logarithmic derivative, we obtain from \eqref{EQ6} that $m(r,g)=O(\log rT(r,w))$. Hence $T(r,g)=S(r,w)$. Multiplying by $w^2$ on both sides of \eqref{EQ6} gives
\begin{equation}\label{EQ6extra}
g(z)w^2=w'^2+Aw'w+Bw+\beta w'+\gamma.
\end{equation}
Then, by eliminating $w'^2$ from \eqref{EQ6extra} and \eqref{HAYMANEQ1} and then using \eqref{EQ6extra} to eliminate the terms $ww''$ and $w'^2$ together with the expressions of $A$ and $B$ in \eqref{EQ5}, we finally obtain
\begin{equation}\label{EQ9}
\left(A'+a A-2b\right)w'=\left(g'+2a g+Ab\right)w-E,
\end{equation}
where $E=B'+2a B+A\alpha+\beta(g-b)$.

\vskip 3pt

\noindent\textbf{Case 3a:} $A'+a A-2b\not\equiv0$.

\vskip 3pt

Since $g'+2a g+Ab$ and $E$ are both small functions of $w$, we have $g'+2a g+Ab\not\equiv0$. For convenience, we denote
\begin{equation}\label{EQ9 fu}
\begin{split}
h_1=\frac{g'+2a g+Ab}{A'+a A-2b}, \qquad h_2=\frac{-E}{A'+a A-2b}.
\end{split}
\end{equation}
Then we have $w'=h_1w+h_2$. Substituting $w'=h_1w+h_2$ into \eqref{EQ6extra}, we have
\begin{equation*}
\left(g-h_1^2-Ah_1\right)w^2-\left(2h_1h_2+\beta h_1+Ah_2+B\right)w-h_2^2-\beta h_2-\gamma=0
\end{equation*}
with coefficients that are small functions of $w$. The Valiron--Mohon'ko theorem implies that all the coefficients of the above equation must vanish identically. In particular, we have $g=h_1^2+Ah_1$. Thus the function $G=g+A^2/4$ satisfies the equation $(G'+2aG)^2=(A'+aA-2b)^2G$. Then we may write $G=H^2$ for a meromorphic function $H$ and obtain $2H'+2aH=A'+aA-2b$. Together with this relation, we substitute $g=H^2-A^2/4$ into the expression of $h_1$ and find that $h_1=H-A/2$. It follows that $h_1$ always satisfies the equation $h_1'+ah_1+b=0$. Since $A'+a A-2b\not\equiv0$ and $g'+2a g+Ab\not\equiv0$, we see that $G\not\equiv0$. Note that $w''=(h_1'+h_1^2)w+h_1h_2+h_2'$. Then $w$ solves \eqref{HAYMANEQ1} if and only if
\begin{equation*}
\left(-h_1h_2+h_2'+ah_2-\alpha-\beta h_1\right)w+h_2^2+\beta h_2+\gamma=0
\end{equation*}
with coefficients being small functions of $w$. So we have $-h_1h_2+h_2'+ah_2-\alpha-\beta{h_1}\equiv0$ and $h_2^2+\beta{h_2}+\gamma\equiv0$. Moreover, since $\gamma\not\equiv0$, we see that $h_2\not\equiv0$ and $h_2+\beta\not\equiv0$, which imply that $h_1,h_2$ are both nonzero rational functions. These results give part~(4) of the theorem.

\vskip 3pt

\noindent\textbf{Case 3b:} $A'+aA-2b\equiv0$.

\vskip 3pt

Since $g'+2a g+Ab$ and $E$ are both small functions of $w$, it follows from \eqref{EQ9} that $g'+2a g+Ab\equiv0$ and $E\equiv0$. Now equation \eqref{EQ6extra} can be rewritten as
\begin{equation}\label{EQ10}
\left(w'+\frac{1}{2}[Aw+\beta]\right)^2=\left(g+\frac{A^2}{4}\right)w^2+\left(\frac{\beta}{2}A-B\right)w+\left(\frac{\beta^2}{4}-\gamma\right).
\end{equation}
Since all poles of $A(z)$ and $a(z)$ are located in the finite disk $D=\{z: |z|<r_0\}$, where $r_0$ is a positive constant, then $h(z):=(\frac{\beta}{2}A-B)e^{\int(\frac{A}{2}+a)dz}$ is well defined and analytic in $\mathbb{C}-\bar{D}$. In $\mathbb{C}-\bar{D}$, the first equation of \eqref{EQ5} yields
\begin{equation}\label{EQ11}
\left(\left[\frac{\beta^2}{4}-\gamma\right]e^{\int(A+2a)dz}\right)'=\frac{\beta}{2}e^{\int(\frac{A}{2}+a)dz}h.
\end{equation}
Together with the second equation in \eqref{EQ5} and the relation $A'+aA-2b\equiv0$, we find that the condition $E\equiv0$ is equivalent to
\begin{equation}\label{EQ12}
h'+a h=\left(g+\frac{A^2}{4}\right)\beta{e^{\int(\frac{A}{2}+a)dz}}
\end{equation}
in $\mathbb{C}-\bar{D}$. Clearly, if $g=-A^2/4$, then $h'+a h=0$ and it follows that $k_1=(\frac{\beta}{2}A-B)e^{\int(\frac{A}{2}+2a)dz}$ is a nonzero constant if $h\not\equiv0$. Moreover, if $\beta^2/4\equiv\gamma$, then from \eqref{EQ11} we see that $h\equiv0$ and then from \eqref{EQ12} that $g=-A^2/4$.

\vskip 3pt

\noindent\textbf{Case 3b(i)}: $g+A^2/4\not\equiv0$.

\vskip 3pt
Equations $A'+a A-2b\equiv0$ and $g'+2a g+Ab\equiv0$ implies that $G=g+A^2/4$ satisfies the equation $G'+2a G=0$. We may write $G=k_1^2e^{-\int 2a dz}$, where $k_1$ is a nonzero constant and $e^{-\int 2a dz}$ is a meromorphic function with at most finitely many zeros and poles. It follows from \eqref{EQ11} and \eqref{EQ12} that
\begin{equation}\label{EQ12a}
\left(\left[\frac{\beta^2}{4}-\gamma\right]e^{\int(A+2a)dz}\right)'=\frac{h}{2k_1^2}(h'+a h)e^{\int 2a dz}=\frac{1}{4k_1^2}\left(h^2e^{\int 2a dz}\right)'.
\end{equation}
Thus, in $\mathbb{C}-\bar{D}$, we may integrate both sides of equation \eqref{EQ12a} to obtain that
\begin{equation}\label{EQ13}
k_2^2=\frac{1}{k_1^2}\left[\frac{1}{4k_1^2}\left(\frac{\beta}{2}A-B\right)^2e^{\int 2a dz}-\left(\frac{\beta^2}{4}-\gamma\right)\right]e^{\int(A+2a)dz}
\end{equation}
is a constant. By analytic continuation, equation \eqref{EQ13} holds for all $z$ in the plane in both of the two cases $k_2=0$ and $k_2\not=0$. In particular, if $k_2\not=0$, then $e^{\int Adz}$ is also a meromorphic function with at most finitely many zeros and poles. Let $u=we^{\int(\frac{A}{2}+a)dz}+\frac{h}{2k_1^2}e^{\int 2a dz}$. By the definition of $h$, in $\mathbb{C}-\bar{D}$ equation \eqref{EQ10} becomes
\begin{equation}\label{EQ14}
(u'-a u)^2=k_1^2e^{-\int 2a dz}\left(u^2-k_2^2e^{\int 2a dz}\right)=k_1^2e^{-\int 2a dz}u^2-k_1^2k_2^2.
\end{equation}

If $k_2\not=0$, then $e^{\int a dz}$ and $e^{\int \frac{A}{2} dz}$ are in general algebroid functions and has at most finitely many algebraically branched points in the plane, as mentioned in the introduction. In this case, we may fix one branch of $u$. The function $v=e^{-\int a dz}u$ leads equation \eqref{EQ14} to
\begin{equation}\label{EQ14fuj}
v'^2e^{\int 2a dz}=k_1^2\left(v^2-k_2^2\right).
\end{equation}
Rewrite \eqref{EQ14fuj} as
\begin{equation*}
\left(v'e^{\int a dz}+k_1v\right)\left(v'e^{\int a dz}-k_1v\right)=-k_1^2k_2^2.
\end{equation*}
Denote $\kappa:=v'e^{\int a dz}+k_1v$. It follows that $v'e^{\int a dz}-k_1v=-k_1^2k_2^2\kappa^{-1}$ and further that
\begin{equation}\label{EQ14fujht}
\begin{split}
v=\frac{1}{2k_1}\left(\kappa+k_1^2k_2^2\kappa^{-1}\right),\qquad
v'=\frac{1}{2}\left(\kappa-k_1^2k_2^2\kappa^{-1}\right)e^{-\int a dz}.
\end{split}
\end{equation}
By taking the derivatives on both sides of the first equation in \eqref{EQ14fujht} and then comparing the resulting equation with the second equation in \eqref{EQ14fujht}, we find that $\kappa'/\kappa=k_1e^{-\int a dz}$ and thus, by integration, $\kappa=k_1k_2\exp(c_1+{k_1\int e^{-\int a dz} dz})$, where $c_1$ is a constant. Therefore, $u=k_2e^{\int a dz}\cosh(c_1+k_1\int e^{-\int a dz} dz)$. Since $u=we^{\int(\frac{A}{2}+a)dz}+\frac{h}{2k_1^2}e^{\int 2a dz}$, we have
\begin{equation}\label{EQ14fujhthil}
\begin{split}
w=k_2 e^{-\int\frac{A}{2}dz}\cosh \left(c_1+k_1\int e^{-\int a dz}dz\right)-\frac{1}{2k_1^2}\left(\frac{\beta}{2}A-B\right)e^{\int 2a dz}.
\end{split}
\end{equation}
Since $w$ and $e^{\int 2a dz}$ are both meromorphic functions and since $\cosh(c_1+k_1\int e^{-\int a dz}dz)$ cannot have any algebraically branched points in the plane, we see from the expression of $w$ in \eqref{EQ14fujhthil} that $e^{-\int\frac{A}{2}dz}$ cannot have algebraically branched points in the plane either. Thus $e^{-\int\frac{A}{2}dz}$ is a meromorphic function. Moreover, if $e^{\int 2a dz}$ is a rational function, then from \eqref{EQ13} we see that $e^{\int \frac{A}{2}dz}$ is also a rational function. If $e^{\int 2a dz}$ is transcendental, since $E\equiv0$, we must have $\beta\equiv0$. Recall that $\gamma\not\equiv0$. If $B\not\equiv0$, then the function $B^2e^{\int 2a dz}/4k_1^2+\gamma$ would have infinitely many zeros by the second main theorem of Yamanoi for rational functions as targets~\cite{yamanoi:05}. Thus by \eqref{EQ13} we must have $B\equiv0$, i.e., $\alpha\equiv0$. It follows that $A+2a=-\gamma'/\gamma$ and, together with the relation $A'+aA-2b=0$, that $2(a'+a^2+b)+(\gamma'/\gamma)'+a(\gamma'/\gamma)=0$. This gives part (5)(a) of the theorem.

If $k_2=0$, then by \eqref{EQ13} we see that $e^{\int 2a dz}$ is a rational function. By \eqref{EQ14}, $u$ satisfies the first order differential equation $u'-a u=\pm k_1e^{-\int a dz}u$ in $\mathbb{C}-\bar{D}$ and thus, by integration, we have $u=c_1e^{\int{(a\pm k_1e^{-\int a dz}})dz}$ in $\mathbb{C}-\bar{D}$, where $c_1$ is a nonzero constant. Since $u=we^{\int(\frac{A}{2}+a)dz}+\frac{h}{2k_1^2}e^{\int 2a dz}$, we thus have, in $\mathbb{C}-\bar{D}$,
\begin{equation*}
\begin{split}
w=c_1e^{\int{\left(-\frac{A}{2}\pm k_1e^{-\int a dz}\right)}dz}-\frac{1}{2k_1^2}\left(\frac{\beta}{2}A-B\right)e^{\int 2a dz}.
\end{split}
\end{equation*}
Since $w$ and $e^{\int 2a dz}$ are both meromorphic functions, by analytic continuation we see that $e^{\int{(-\frac{A}{2}\pm k_1e^{-\int a dz})}dz}$ is also a meromorphic function. This implies that $e^{-\int a dz}$ cannot have any branched points in the plane and thus is a rational function. These results give Part (5)(b) of the theorem.

\vskip 3pt

\noindent\textbf{Case 3b(ii)}: $g+A^2/4\equiv0$, $h\not\equiv0$.

\vskip 3pt

In this case $h'+ah=0$. We may write $h=k_1e^{-\int a dz}$ and, recalling $h=(\frac{\beta}{2}A-B)e^{\int(\frac{A}{2}+a)dz}$ in $\mathbb{C}-\bar{D}$, it follows that $k_1=(\frac{\beta}{2}A-B)e^{\int(\frac{A}{2}+2a)dz}$ is a nonzero constant. By analytic continuation, $e^{\int(\frac{A}{2}+2a)dz}$ is a rational function. Then from \eqref{EQ11} we have
\begin{equation}\label{EQ11 a}
\left(\frac{\beta^2}{4}-\gamma\right)'+(A+2a)\left(\frac{\beta^2}{4}-\gamma\right)=\frac{\beta}{2}\left(\frac{\beta}{2}A-B\right).
\end{equation}
By equation \eqref{EQ10}, we may write $\left(\frac{\beta}{2}A-B\right)w+\left(\frac{\beta^2}{4}-\gamma\right)=\left(\frac{\beta}{2}A-B\right)^2\frac{H^2}{4}$ for a meromorphic function $H$. Then we have
\begin{equation*}
\begin{split}
w&=\left(\frac{\beta}{2}A-B\right)\frac{H^2}{4}-\frac{\frac{\beta^2}{4}-\gamma}{\frac{\beta}{2}A-B},\\
w'&=\left(\frac{\beta}{2}A-B\right)'\frac{H^2}{4}+\left(\frac{\beta}{2}A-B\right)\frac{HH'}{2}-\left(\frac{\frac{\beta^2}{4}-\gamma}{\frac{\beta}{2}A-B}\right)'.
\end{split}
\end{equation*}
Substituting the above two expressions for $w$ into equation \eqref{EQ10} together with the relation $(\frac{\beta}{2}A-B)'+(\frac{A}{2}+2a)(\frac{\beta}{2}A-B)=0$ and equation \eqref{EQ11 a}, we finally obtain that $H'=aH+1$. These results give part (5)(c) of the theorem.

\vskip 3pt

\noindent\textbf{Case 3b(iii)}: $g+A^2/4\equiv0$, $h\equiv0$.

\vskip 3pt

From \eqref{EQ10} and \eqref{EQ11}, it follows that $w'+\frac{1}{2}(Aw+\beta)=k_1e^{-\int(\frac{A}{2}+a)dz}$, where $k_1^2=(\frac{\beta^2}{4}-\gamma)e^{\int(A+2a)dz}$ is a constant. If $k_1\not=0$, we see that $e^{\int(\frac{A}{2}+a)dz}$ is a rational function. These results give part~(5)(d) of the theorem. Otherwise, we have $k_1=0$ and it follows that $w'+\frac{1}{2}(Aw+\beta)=0$. These results give part (5)(e) of the theorem and also complete the proof.

\end{proof}




\end{document}